\documentclass[12pt]{article}
\usepackage[utf8]{inputenc}
\usepackage{amsmath,amssymb,amsthm}
\usepackage{xcolor}
\renewcommand\le{\leqslant}
\renewcommand\ge{\geqslant}
\newcommand\E{\mathsf E}
\renewcommand\P{\mathsf P}
\newcommand\R{\mathbb R}
\newcommand\eps{\varepsilon}
\newcommand\avg{\mathrm{avg}}

\DeclareMathOperator\sign{sign}

\DeclareMathOperator\tr{tr}
\DeclareMathOperator\Law{Law}

\title{Widths and rigidity of unconditional sets and random vectors}
\author{Yuri Malykhin\thanks{Steklov Mathematical Institute, Lomonosov Moscow
State University; malykhin-yuri@yandex.ru},
Konstantin Ryutin\thanks{Lomonosov Moscow State University, Moscow Center of
Fundamental and Applied Mathematics; kriutin@yahoo.com}}

\newtheorem{statement}{Statement}
\newtheorem*{corollary}{Corollary}
\newtheorem{theorem}{Theorem}
\newtheorem{lemma}{Lemma}
\newtheorem*{remark}{Remark}

\begin{document}
\maketitle
\abstract{
    We prove that any unconditional set in $\R^N$ that is invariant under cyclic
    shifts of coordinates is rigid in $\ell_q^N$, $1\le q\le 2$, i.e. it
    can not be well approximated by linear spaces of dimension essentially
    smaller than $N$. We apply the approach of E.D.~Gluskin to the setting of
    averaged Kolmogorov widths of unconditional random vectors
    or vectors  of independent mean zero random variables,
    and prove their rigidity. These results are obtained using  a general
    lower bound for the averaged Kolmogorov width via weak moments of
    biorthogonal random vector.
    This paper continues the study of the rigidity initiated by the first
    author.

    We also provide several corollaries including new bounds for Kolmogorov
    widths of mixed norm balls.  
}

\section{Introduction}
This paper continues the study of the rigidity of random vectors and sets
initiated by the first author in~\cite{M24}. 
This research is related to the notion of matrix rigidity developed within
Complexity Theory, see~\cite{L08,AW17,M22}.

Let us say (informally) that an
$N$-dimensional set is \textit{rigid} if it can not be well
approximated by linear spaces of dimension essentially smaller than $N$.
The typical example of an extremely rigid set is given by any finite dimensional section of a unit ball in a normed space:
Tikhomirov's theorem states that any nontrivial subspace does not
approximate this section better than a single point $0$.  

The error of the approximation in a normed space $X$ by subspaces of given
dimension is captured by the quantity called the Kolmogorov width of a set $K\subset
X$:
$$
d_n(K,X) := \inf_{Q_n\subset X}\sup_{x\in K}\rho(x,Q_n)_X,
$$
where
$$
\rho(x,Q)_X := \inf_{y\in Q}\|x-y\|_X,
$$
and $\inf_{Q_n}$ is over all linear subspaces in $X$ of dimension at most $n$. 

We consider the usual spaces $\ell_q^N$ with the norm
$\|x\|_q := (\sum_{k=1}^N |x_k|^q)^{1/q}$. Let $B_q^N$ be the unit ball of this
space. In the case $q=2$ this is the usual euclidean norm and we
write it as $|x|$. Throughout the paper $q'$ denotes the conjugate of $q$ i.e. $1/q+1/q'=1$.

\paragraph{Main results.}
Throughout the paper we freely use the informal notion of rigidity and  interpret  formal statements  as the rigidity  of a corresponding  object (a sequence of sets,  a random vector etc). Let us give  several examples.
We recall some estimates for the widths of finite-dimensional balls (see e.g.
~\cite{LGM,Tikh87,P85}).
It is well known that the octahedron $B_1^N$ is well approximated in $\ell_q^N$
for $q>2$ by low-dimensional spaces:
$$
d_n(B_1^N,\ell_q^N) \le C(q)N^{-\delta}\quad\mbox{if $n\ge N^{1-\delta}$,
$\delta=\delta(q)>0$}.
$$
However, if $1\le q\le 2$, then all the balls $B_p^N$ are rigid in $\ell_q^N$ in the following
sense:
$$
d_n(B_p^N,\ell_q^N) \ge c(\eps)d_0(B_p^N,\ell_q^N) = c(\eps)\sup_{x\in B_p^N}\|x\|_q,
\quad \mbox{if }n\le N(1-\eps).
$$
Indeed, if $p\le q\le 2$, then $d_0(B_p^N,\ell_q^N)=1$, but
$$
d_n(B_p^N,\ell_q^N) \ge d_n(B_1^N,\ell_2^N) = \sqrt{1-n/N} \ge \eps^{1/2}.
$$
If $p\ge q$, then it is known that
$$
d_n(B_p^N,\ell_q^N) = (N-n)^{\frac1q-\frac1p} \ge
    \eps^{\frac1q-\frac1p}N^{\frac1q-\frac1p} =
    \eps^{\frac1q-\frac1p}d_0(B_p^N,\ell_q^N).
    $$
E.D.~Gluskin~\cite{G87} proved that the set $V_k^N = B_\infty^N \cap kB_1^N$ is
 rigid:
\begin{equation}
    \label{glus_vkn}
d_n(V_k^N,\ell_q^N) \ge c(q)k^{1/q},\quad\mbox{if $n\le N/2$ and $1\le q\le 2$.}
\end{equation}
Later this method was used by E.M. Galeev, A.D. Izaak, A.A. Vasilieva \cite{Gal90,Gal95,Gal11,Iza90,Vas23,Vas13} to obtain lower bounds for the widths of balls in mixed norms and  cartesian products of several copies of $V_k^m.$ 
We also develop   these ideas of E.D. Gluskin and obtain more general results.

Let us say that a set $K\subset\R^N$ is unconditional if it is invariant under
arbitrary changes of signs of coordinates:
\begin{equation}
    \label{uncond_body}
(x_1,\ldots,x_N)\in K \quad \Rightarrow \quad (\pm x_1,\ldots, \pm x_N)\in K.
\end{equation}

Our technique allows to prove the rigidity for a large family of  unconditional convex
bodies in $\ell_q^N$. One of the main examples is the following.

\begin{theorem}
\label{th_uncond_cyclic}
    Let $K$ be an unconditional set in $\R^N$ that is also invariant under
    cyclic shifts of coordinates. Then for $1\le q\le 2$, any $\eps\in(0,1)$ we
    have
    $$
    d_n(K,\ell_q^N) \ge c(q,\eps)\sup_{x\in K}\|x\|_q,\quad\mbox{if $n\le
    N(1-\eps)$.}
    $$
\end{theorem}
Here and later the letters $c,C,c_q,c(q,\eps),\ldots$ denote positive quantities that
may depend only on the indicated parameters and their values can vary from
line to line.

An important example covered by Theorem  \ref{th_uncond_cyclic} is any unconditional  set invariant under all  permutations    of coordinates. In fact all we need is the invariance under  some transitive subgroup of permutations.

It seems interesting to understand the dependence on $q$ in
Theorem~\ref{th_uncond_cyclic}.
We remark that~\eqref{glus_vkn} was improved in~\cite{MR17} to the form
$d_{N/2}(V_k^N,\ell_q^N)\ge ck^{1/q}$, $1\le q\le 2$, with an absolute
constant $c$.

A similar  geometric problem was considered by B. Green, A. Sah, M. Sawhney,
Y.Zhao in \cite{G,SSZ}. They studied the widths of an orbit $X=Gx$ of 
some fixed unit vector under a finite subgroup $G$ of the orthogonal group  $\mathrm{O}(N)$. 
They established that $d_{N-k}(X,\ell_2^N)\le c\log^{-1/2}(N/k)$
for any such set
when $1\le k \le N/\log{(3N)}^C$. This estimate is sharp. 
We obtain lower estimate for the Kolmogorov width in a particular situation of
the orbit of a single vector under the group of cyclic shifts of coordinates
and arbitrary changes of their signs~$\eqref{orbit}$ in $\ell_q^N$, but for any
$1\le q \le 2$.

Let us explain the main steps  of the proof and give relevant definitions.
We freely use   probabilistic notation.  
Let $\xi$ be a (Borel) random vector in a normed space $X$.
Define the $p$-averaged Kolmogorov width, $1\le p<\infty$, of $\xi$ in $X$ as
$$
d^{\mathrm{avg}}_n(\xi, X)_p := \inf_{Q_n\subset X} (\E \rho(\xi,Q_n)_X^p)^{1/p}.
$$
For the details and the history of this notion we refer to~\cite{M24}.

We call random vectors $\xi=(\xi_1,\ldots,\xi_N)$ and
$\eta=(\eta_1,\ldots,\eta_N)$ biorthogonal, when  $\E\xi_i\eta_i=1$ and
$\E\xi_i\eta_j=0$ for all $i\ne j$.
Recall also the definition of the weak $p$-moment of a random vector  $\eta$ in
$\R^N$:
$$
\sigma_p(\eta) := \sup_{v\in\R^N,|v|=1}\left(\E|\langle \eta,v\rangle|^p\right)^{1/p}.
$$

\begin{theorem}
    \label{th_sp}
         Let $\xi$ and $\eta$ be biorthogonal random vectors in $\R^N$. Then for any  $n<N$ we have
$$
    d_n^\avg(\xi,\ell_q^N)_q \ge \frac{(N-n)^{1/q}}{\sigma_{q'}(\eta)},\quad
    1<q\le 2,
$$
$$
    d_n^\avg(\xi,\ell_q^N)_q \ge
    \frac{(N-n)^{1/2}N^{1/q-1/2}}{\sigma_{q'}(\eta)},\quad
    2\le q<\infty.
$$
\end{theorem}

One can bound weak moments for vectors that are unconditional in $L_p$, or that
satisfy a weaker RUD (Random Unconditional Divergence) property~\eqref{rud}, see
Lemma~\ref{lem_rud}.
We learned this notion from S.V.~Astashkin (see~\cite{AC18});
he also suggested to use it in this context.
The crucial fact used in Lemma~\ref{lem_rud} is the type-2 property of the
$L_{q'}$ space when  $q'\ge 2$.

In order to obtain the RUD property for the biorthogonal vector $\eta$ we
impose more restrictive requirements on $\xi$.

We say that a random vector  $\xi=(\xi_1,\ldots,\xi_N)$ in $\R^N$ has the
property  ($\star$) if either the distribution  $\xi$ is unconditional, i.e.
$\Law(\xi_1,\ldots,\xi_N)=\Law(\pm\xi_1,\ldots,\pm\xi_N)$ for any choice of
signs; or random variables  $\xi_1,\ldots,\xi_N$ are independent and
$\E\xi_1=\cdots=\E\xi_N=0$.
\begin{theorem}
    \label{thm_uncond}
    Let $1\le q\le 2$, we assume that random vector  $\xi=(\xi_1,\ldots,\xi_N)$ satisfies
    ($\star$) and  $1\le\E|\xi_i|^q < \infty$ for $i=1,\ldots,N$.
    Then for any $\eps\in(0,1)$ we have
$$
    d_n^{\mathrm{avg}}(\xi, \ell_q^N)_q \ge c(q,\eps)N^{1/q},
    \quad\mbox{if $n\le N(1-\eps)$.}
$$
\end{theorem}
The case of $q=1$ is covered in~\cite{M24}, see Theorem 3.1 and Corollary 3.1
there; the proof uses different techniques (VC-dimension). 
The case $1<q\le 2$ is contained in Statement~\ref{stm_uncond} below.

Now the proof of the Theorem \ref{th_uncond_cyclic} is a one--liner. Indeed, consider any $x\in K$ and
let $\xi$ be the random point that is obtained from $x$ by randomly changing signs
of coordinates and randomly shifting them. Then $\xi$ has the ($\star$) property
and $\E|\xi_i|^q=N^{-1}\|x\|_q^q$, so
\begin{equation}
    \label{proof_main_th}
d_n(K,\ell_q^N) \ge d_n^\avg(\xi,\ell_q^N)_q \ge c(q,\eps)\|x\|_q.
\end{equation}

Other applications  of Theorem~\ref{th_sp} and
Theorem~\ref{thm_uncond}
are given in Section~\ref{sec_examples}.

\paragraph{Rigidity of finite function systems.}
We can treat a random vector $\xi=(\xi_1,\dots, \xi_N)$
composed of random variables $\xi_i\colon\Omega\to\R$ as a system of elements from the normed space, say when $\xi_i\in L^q(\Omega).$
We can reformulate our results in terms of widths of finite sets in $L^q$ using the
following simple equality (see~\cite[Section 2]{M24}): 
$$
d_n^{\mathrm{avg}}(\xi,\ell_q^N)_q =
N^{1/q} d_n^{\mathrm{avg}}(\{\xi_1,\ldots,\xi_N\},L_q(\Omega))_q.
$$
Here the width on the right side is the average width of a set of $N$
 elements of
$L_q(\Omega)$. The average width of a finite set is defined as an average width
of a random point in that set:
$$
d^{\mathrm{avg}}_n(\{x_1,\ldots,x_N\}, X)_q := \inf_{Q_n\subset X}\left(\frac1N
\sum_{i=1}^N \rho(x_i,Q_n)_X^q\right)^{1/q}.
$$

So, we can give a functional--theoretic interpretation of
Theorem~\ref{th_sp}. If systems of functions $\{f_1,\ldots,f_N\}$
and $\{g_1,\ldots,g_N\}$ are biorthogonal, and all the norms $\|\sum
a_ig_i\|_{q'}$, $\sum a_i^2=1$, are bounded, then the set $\{f_1,\ldots,f_N\}$
is rigid in $L_q$. Therefore we generalize Theorem 4.2 from~\cite{M24} on
$L_q$-rigidity of  orthonormal systems with this property
(they are called $S_{q'}$-systems).

Theorem~\ref{thm_uncond} can be formulated in purely probabilistic terms:
if $1\le q\le 2$, a random vector $\xi=(\xi_1,\ldots,\xi_N)\colon\Omega\to\R^N$ satisfies 
$(\star)$ and $\E|\xi_i|^q \ge 1$, then for any random variables
$\eta_1,\ldots,\eta_n\in L_q(\Omega)$, $n\le N(1-\eps)$, and any coefficients 
$(a_{i,j})$, we have
$$
\sum_{i=1}^N \E|\xi_i - \sum_{j=1}^n a_{i,j}\eta_j|^q \ge c(q,\eps)N.
$$

In the case $q>2$ good approximation is possible. It was proved in~\cite{M24}
that for any $q>2$ there exists a sequence $\xi_1^{(N)},\ldots,\xi_N^{(N)}$ of
independent symmetric random variables, $\E|\xi_i^{(N)}|^q=1$, such
that $d_n(\{\xi_1^{(N)},\ldots,\xi_N^{(N)}\},L_q)=o(1)$, $n=o(N)$. One can
generalize this result to the case of  Orlicz spaces.
It seems interesting to characterize  symmetric  spaces where independent zero-mean
random variables are rigid.

The next Section is devoted to the proofs of the results stated above. 
We have to  deal only with the case $q>1$.

\section{Proofs of the main results}

\begin{lemma}
    \label{lem_dual}
    Let $\xi$ be a random vector in a normed space $X$ and $Q\subset X$ be a
    subspace. For any random vector $\zeta$ in $X^*$, such that
    $\P(\zeta\in Q^\perp)=1$, and for any $1\le p<\infty$, we have
    $$
    (\E\rho^p(\xi,Q)_X)^{1/p} \ge
    \frac{\E\langle\xi,\zeta\rangle}{(\E\|\zeta\|_{X^*}^{p'})^{1/p'}}.
    $$
    We assume that all the expectations are finite; $Q^\perp=\{x^*\in
    X^*:\langle x^*,x\rangle=0\;\;\forall x\in Q\}$.
\end{lemma}

\begin{proof}
    It is obvious that $\rho(x,Q)_X \ge \langle x,z\rangle / \|z\|_{X^*}$ for
    any $x\in X$ and $z\in Q^\perp$.
    We use this inequality for $x:=\xi$ and $z:=\zeta$ and obtain
    $$
    \E\langle\xi,\zeta\rangle =
    \E\frac{\langle\xi,\zeta\rangle}{\|\zeta\|_{X^*}}\|\zeta\|_{X^*}\le
    \E\rho(\xi,Q)_X\|\zeta\|_{X^*}.
    $$
    Finally, we apply H\"{o}lder inequality and get the required bound:
    $$
        \E\langle\xi,\zeta\rangle \le
        (\E\rho^p(\xi,Q)_p)^{1/p}(\E\|\zeta\|_{X^*}^{p'})^{1/p'}.
    $$
\end{proof}

Let us prove Theorem~\ref{th_sp}.

\begin{proof}
    We consider some  $n$-dimensional subspace $Q_n\subset\R^N$.
    Let $P$ be the orthoprojector onto $Q_n^\perp$, and $v_k := Pe_k$.
    We apply Lemma~\ref{lem_dual} with $\xi$ and $\zeta := P\eta$:
    \begin{equation}
        \label{dual_low}
        (\E\rho^q(\xi,Q_n)_q)^{1/q}
        \ge \frac{\E\langle\xi,P\eta\rangle}{(\E\|P\eta\|_{q'}^{q'})^{1/q'}}.
    \end{equation}

    We evaluate the numerator using biorthogonality:
    $$
    \E \langle \xi,P\eta\rangle = \E\langle \sum_{i=1}^N \xi_i e_i, \sum_{i=1}^N
    \eta_i v_i\rangle = \sum_{i=1}^N \langle e_i,v_i\rangle = \tr P = N-n.
    $$

    As to the denominator of~\eqref{dual_low}:
    $$
    \E\|P\eta\|_{q'}^{q'} = 
    \E\sum_{k=1}^N |\langle P\eta,e_k\rangle|^{q'} = 
    \E\sum_{k=1}^N |\langle \eta,v_k\rangle|^{q'} \le
    \sigma_{p'}(\eta)^{q'}\sum_{k=1}^N|v_k|^{q'}.
    $$
    We estimate the sum when $q\le 2$ as
    $\sum_{k=1}^N|v_k|^{q'} \le \sum_{k=1}^N|v_k|^2 = N-n$, since $|v_k|\le 1.$ 
    In the case $q>2$:
    $$
    (\sum_{k=1}^N|v_k|^{q'})^{1/q'} \le N^{1/q'-1/2}(\sum_{k=1}^N|v_k|^2)^{1/2} =
    N^{1/q'-1/2}(N-n)^{1/2}.
    $$
    In both cases we obtain the required estimates.
\end{proof}

We can estimate $\sigma_p$ using the RUD-property.  This definition works for any normed space $X$. We say that
a system $x_1,\ldots,x_N\in X$ has the RUD  property with a constant $D$, if for
any coefficients $c_k$ we have
\begin{equation}
    \label{rud}
\|\sum_{k=1}^N c_k x_k \|_X \le D\cdot \E_\eps \|\sum_{k=1}^N \eps_k c_k x_k \|_X.
\end{equation}
Here $\E_\eps$ is the mathematical expectation with respect to independent
random signs $\eps_k$, $\P(\eps_k=1)=\P(\eps_k=-1)=1/2$.

\begin{lemma}
    \label{lem_rud}
    Let $2\le p<\infty$ and assume that random variables
    $\eta_1,\ldots,\eta_N\colon\Omega\to\R$
    have the RUD property in $L_p(\Omega)$ with constant $D$. Then
\begin{equation}
    \label{glus}
    \sigma_p(\eta) \le C_p D \max_{1\le i\le N}(\E|\eta_i|^p)^{1/p}.
\end{equation}
\end{lemma}
\begin{proof}
    We know that the space $L_p$, $p\ge 2$ has type $2$. Therefore, for any unit
    vector $v$, we have
    \begin{multline*}
    \|\langle \eta,v\rangle\|_{L_p}
    = \|\sum_{i=1}^N v_i\eta_i\|_{L_p} 
    \le D \E_\eps\|\sum_{i=1}^N \eps_i\eta_i v_i\|_{L_p} \le \\
    \le C_p D (\sum_{i=1}^N v_i^2 \|\eta_i\|_{L_p}^2)^{1/2} \le C_p D \max\|\eta_i\|_{L_p}.
    \end{multline*}
\end{proof}

Recall the property ($\star$) for a random vector $\xi=(\xi_1,\ldots,\xi_N)$:
either the distribution  $\xi$ is unconditional, i.e.
$\Law(\xi_1,\ldots,\xi_N)=\Law(\pm\xi_1,\ldots,\pm\xi_N)$ for any choice of
signs or random variables  $\xi_1,\ldots,\xi_N$ are independent and their
mathematical expectations are zero.

\begin{statement}
    \label{stm_uncond}
    Let $1<q\le 2$ and random vector  $\xi=(\xi_1,\ldots,\xi_N)$ satisfies
    ($\star$), assume that  $1\le \E|\xi_i|^q< \infty$ and $i=1,\ldots,N$.
    Then for any $n<N$ we have
$$
    d_n^{\mathrm{avg}}(\xi, \ell_q^N)_q \ge c_q(N-n)^{1/q}.
$$
\end{statement}

\begin{proof}
    Since the mapping  $(x_1,\ldots,x_N)\mapsto(a_1x_1,\ldots,a_Nx_N)$
    is linear and does not increase distances when $\|a\|_\infty\le 1$, we can
    assume that $\E|\xi_i|^q=1$ for $i=1,\ldots,N$.

    Let $\phi(t):=|t|^{q-1}\sign t$. Consider the random vector  $\eta$, given
    by $\eta_i = \phi(\xi_i)-\E\phi(\xi_i)$. It is clear that  $\xi$ and $\eta$
    are biorthogonal. We show that  $\eta$ also satisfies property  ($\star$).
    Really, the independence of  $\xi_i$ leads to the independence of $\eta_i$,
    and $\E\eta_i=0$. For unconditional $\xi$ since our function $\phi$ is odd
    we have
    $\E\phi(\xi_i)=0$ and the random vector  $(\phi(\xi_i))_{i=1}^N$ is also unconditional.

    We remark that
    $$
    \|\eta_i\|_{L_{q'}} \le \|\phi(\xi_i)\|_{L_{q'}} +
    \|\phi(\xi_i)\|_{L_1} \le 2\|\phi(\xi_i)\|_{L_{q'}} =
    2\|\xi_i\|_{L_q}^{q/q'}.
    $$
    Since $\eta$ satisfies ($\star$), it is unconditional in $L_{q'}$ i.e. 
    $$
    \max_{\eps_1,\ldots,\eps_N\in\{-1,1\}}\|\sum_{k=1}^N \eps_k c_k \eta_k \|_{L_{q'}}
    \le C(q) \, \min_{\eps_1,\ldots,\eps_N\in\{-1,1\}}\|\sum_{k=1}^N \eps_k c_k \eta_k \|_{L_{q'}}.
$$
    For independent zero-mean $\eta_i$ this is a well--known property; see,
    e.g., \cite{AS10}[\S3].
    Therefore $\eta$ has the RUD-property in $L_{q'}$ and we can use~\eqref{glus}:
    $$
    \sigma_{q'}(\eta)^{q'} \le C_1(q)\max\E|\eta_i|^{q'} \le C_2(q)\max\E|\xi_i|^q = C_2(q).
$$
Finally, we apply Theorem~\ref{th_sp}.
\end{proof}

Consider the group  $G :=
\mathbb Z_2^N\times \mathbb Z_N$, that acts on  $\R^N$ via changes of signs
and cyclic shifts of coordinates:
$$
g x = (\tau,k) x = ((-1)^{\tau_1}x_{k+1}, \ldots, (-1)^{\tau_N}x_{k+N})
$$
(the indices are added mod $N$).
Let $\mathbf{g}$ be the uniformly distributed random element of  $G$. 

\begin{corollary}
    Let $1<q\le 2$, $N\in\mathbb N$.
    For any $x\in\R^N$ and $n<N$ we have:
    \begin{equation}
    \label{orbit}
        d_n^{\mathrm{avg}}(\mathbf{g}x,\ell_q^N)_q\ge c(q)(1-n/N)^{1/q}\|x\|_q.
    \end{equation}
    Moreover, for any $n$-dimensional subspace $Q_n\subset\R^N$,
    $$
        \P\left\{\rho(\mathbf{g}x,Q_n)_q \ge c_1(q)(1-n/N)^{1/q}\|x\|_q\right\}
        \ge c_2(q)(1-n/N).
    $$
\end{corollary}

\begin{proof}
    The random vector $\xi=\mathbf{g}x$ is unconditional and 
    $\E|\xi_i|^q=N^{-1}\|x\|_q^q$, therefore we can apply
    Statement~\ref{stm_uncond} to obtain the first inequality.
    The second inequality uses the
    first one and the fact that $\rho(\mathbf{g}x,Q_n)_q\le\|x\|_q$.
\end{proof}

For $q>1$ we can restate Theorem~\ref{th_uncond_cyclic} in the following way.
\begin{statement}
    Let $K$ be an unconditional set in $\R^N$ that is also invariant under
    cyclic shifts of coordinates. Then for $1<q\le 2$, any $n<N$ we have
    $$
    d_n(K,\ell_q^N) \ge c(q)\cdot\left(1-\frac{n}{N}\right)^{1/q}\sup_{x\in
    K}\|x\|_q.
    $$
\end{statement}

\begin{proof}
    If $K$ contains some vector $x$, then  $d_n(K,\ell_q^N)\ge
    d_n^{\mathrm{avg}}(\mathbf{g}x,\ell_q^N)_q$.
\end{proof}

\section{Examples of rigid sets}
\label{sec_examples}

\paragraph{Transitive actions.}
In Theorem~\ref{th_uncond_cyclic} we considered cyclic shifts of coordinates.
One can take any transitive group of permutations instead.

Suppose that $H$ is a subgroup of $S_N$. It acts in a usual way on the set $\{1,\ldots,N\}$. Recall that the
action is called transitive is for any $i,j\in\{1,\ldots,N\}$ there is $h\in H$
such that $h(i)=j$.  We have the following generalization of
Theorem~\ref{th_uncond_cyclic}.
\begin{statement}
    Let $H$ be a group that transitively acts on the set of $N$
    coordinates. Let $K$ be an unconditional set in $\R^N$ that is also
    invariant under the action of $H$, i.e. for all $\pi\in H$,
    $$
    (x_1,\ldots,x_N)\in K \quad \Rightarrow \quad (\pm x_{\pi(1)},\ldots, \pm
    x_{\pi(N)})\in K.
    $$
    Then for $1\le q\le 2$ the set $K$ is rigid in $\ell_q^N$, i.e.
    $$
    d_n(K,\ell_q^N) \ge c(q,\eps)\sup_{x\in K}\|x\|_q,\quad\mbox{if $n\le
    N(1-\eps)$.}
    $$
\end{statement}
The proof is the same as for Theorem~\ref{th_uncond_cyclic},
see~\eqref{proof_main_th}.


\paragraph{Mixed norms.}
Consider the space $\R^N$; assume that the set of the coordinates
is split into $b$ blocks, each of size $s$ (so, $N=sb$).
Given a vector $x=(x_i)\in\R^N$, we denote by $x[j]\in\R^s$ the vector of coordinates
of $x$ from the $j$-th block, i.e. $x[j] := (x_i)_{s(j-1)<i\le sj}$.
Recall the notion of a mixed norm. Let $\Phi$ and $\Psi$ be some
norms on $\R^s$ and $\R^b$, correspondingly.
Then the mixed norm of a vector $x\in\R^N$ is defined as
\begin{equation}
    \label{mixed_norm}
\|x\| := \Psi(z),\quad z:=(\Phi(x[1]),\ldots,\Phi(x[b])).
\end{equation}

Recall that a norm is called symmetric if its unit ball is unconditional and
invariant under arbitrary permutations of coordinates. We prove that the ball of the mix of
symmetric norms is rigid in $\ell_q^N$ for $q\le 2$.
%
\begin{statement}
    \label{stm_mixed_product}
    Let $s,b\in\mathbb N$, $N=sb$; $\Phi$ and $\Psi$ be some symmetric norms in
    $\R^s$ and $\R^b$, correspondingly. Let $B\subset\R^N$ be the unit ball of the mixed
    norm~\eqref{mixed_norm}. Then for any $1\le q\le 2$ and any $n\le
    N(1-\eps)$, $\eps\in(0,1)$, we have
    $$
    d_n(B,\ell_q^N) \ge c(q,\eps)\sup_{x\in B}\|x\|_q =
    c(q,\eps)\sup_{\Phi(y)\le 1}\|y\|_q\sup_{\Psi(z)\le 1}\|z\|_q.
    $$
\end{statement}

\begin{proof}
    We will construct an unconditional random vector $\xi$ in $B$. Let $y$
    be a vector with $\Phi(y)\le 1$ having maximal $\ell_q^s$-norm.
    For each block we consider the vector $\eta^j := \mathbf{g}^jy$, where
    $\mathbf{g}^j$ is a random element of the group $G$ of cyclic shifts and
    sign changes on $\R^s$.
    We also take a vector $z$, $\Psi(z)\le 1$, with the maximal $\ell_q^b$-norm and
    define $\zeta=(\zeta_1,\dots,\zeta_b) := \mathbf{g}z$ analogously. We assume that $\mathbf{g}$ and
    all $\mathbf{g}^j$, $j=1,\ldots,b$, are independent.

    Define  random vector $\xi=(\xi_1,\dots, \xi_N)=(\xi[1],\dots,\xi[b])$ by the equation
    $$
    \xi[j] = \zeta_j \eta^j,\quad j=1,\ldots,b.
    $$
    It is clear that $\xi$ is unconditional. Moreover, for any coordinate in the
    $j$-th block we have, using the independence,
    $$
    \E|\xi_i|^q = \E|\zeta_j|^q\,\E|\eta^j_i|^q =
    (b^{-1}\|z\|_q^q)\cdot(s^{-1}\|y\|_q^q).
    $$
    It remains to apply Theorem~\ref{thm_uncond}.
\end{proof}

A remarkable example of a mixed norm is the mix of $\ell_{q_1}^s$ and
$\ell_{q_2}^{b}$:
\begin{equation}\label{mixed_lp_norm}
\|x\|_{q_1,q_2} := (\sum_{j=1}^b (\sum_{s(j-1)<i\le sj}|x_i|^{q_1})^{q_2/q_1})^{1/q_2}
\end{equation}
(with the usual modification for $q_1=\infty$ or $q_2=\infty$).
The space with this norm is
denoted as $\ell_{q_1,q_2}^{s,b}$ and its unit ball is $B_{q_1,q_2}^{s,b}$, some authors prefer the notation $\ell_{q_2}^b(\ell_{q_1}^s)$ but we decided to keep the notation used in \cite{Vas24,Gal90,Iza90}. 
For the overview of the results on the widths of the balls in mixed norms, see~\cite{Vas24}.

  We have the following  \begin{corollary}
    \label{stm_mixed_lp}
    Let $1\le p_1,p_2\le\infty$; $1\le q\le 2$; $N=sb$. Then for $n\le
    N(1-\eps)$, $\eps\in(0,1)$, we have
    $$
    c(q,\eps)s^{(\frac1q-\frac1{p_1})_+}b^{(\frac1q-\frac1{p_2})_+}
    \le d_n(B_{p_1,p_2}^{s,b},\ell_q^N)
    \le s^{(\frac1q-\frac1{p_1})_+}b^{(\frac1q-\frac1{p_2})_+}.
    $$
    \end{corollary}
In our terminology, the balls $B_{p_1,p_2}^{s,b}$ are rigid in $\ell_q^N$, $q\le
2$, for all $1\le p_1,p_2\le\infty$.
In fact, we can prove   upper and lower bounds for the widths of $B_{p_1,p_2}^{s,b}$  in $\ell_{q_1,q_2}^{s,b}$.
 
\begin{remark}
    1. Assume that a tuple $(p_1,p_2,q_1,q_2)$, where $1\le p_1,p_2,q_1,q_2\le\infty$,
    satisfies two conditions: (i) $p_1\ge q_1$ or $q_1\le 2$; (ii) $p_2\ge q_2$ or
    $q_2\le 2$, and does not fall into the ``exceptional case''

    \begin{equation}
        \label{exception}
        q_1<\min(p_1,q_2)\mbox{ and }p_2<q_2\le 2.
    \end{equation}
    Then, for all $s,b\in\mathbb N$, $N:=sb$, and $\eps>0$, we have
  $$
        \label{mixed_generic}
        d_{N/2}(B_{p_1,p_2}^{s,b},\ell_{q_1,q_2}^{s,b})
        \ge c(q_1,q_2)d_0(B_{p_1,p_2}^{s,b},\ell_{q_1,q_2}^{s,b}).
$$

    2. For any other tuple $(p_1,p_2,q_1,q_2)$ there is some   $\gamma>0$ such that 
    $$
    d_{N/2}(B_{p_1,p_2}^{s,s},\ell_{q_1,q_2}^{s,s})
    \le N^{-\gamma} d_0(B_{p_1,p_2}^{s,s},\ell_{q_1,q_2}^{s,s})
    $$
   holds for any   $N=s^2,s\in \mathbb{N}$.
   \end{remark}

The part 1 of this Remark uses standard facts about
the widths of $B_p^N$, Statement~\ref{stm_mixed_product} and the estimate
from~\cite{MR17}:
$$
    d_{N/2}(B_{1,\infty}^{s,b},\ell_{2,1}^{s,b})\ge cb.
$$

 It seems interesting that a mix of two rigid sets can be non rigid in the case of mixed norms, say 
    $$
    cs^{1/2} \le d_{s(s-1)/2}(B_{\infty,1}^{s,s},\ell_{1,2}) \le s^{1/2},
    $$
    but $d_0(B_{\infty,1}^{s,s},\ell_{1,2}) = s$.

We can interpret the part 1 of the Remark as the rigidity in mixed norms.
We will elaborate on this and give the
details of the proof in the forthcoming paper.

 We mention another generalization of Theorem \ref{th_uncond_cyclic}. Let $1\le
 q \le 2$ and each of the sets $K_1\subset \R^{s_1},\dots, K_b\subset
 \R^{s_b},N=s_1+\dots +s_b$ satisfies the assumptions of Theorem
 \ref{th_uncond_cyclic}.  We assume that  $\sup_{x\in K_j} \|x\|_qs_j^{-1/q}\ge
 1$ for $1\le j \le b $. Then $d_n(K_1\times \dots\times K_b,\ell_q^N)\ge
 c(q,\eps)N^{1/q}$ for any $n\le N(1-\eps).$ This result generalizes one lemma
 by Vasil'eva \cite{Vas23}.

\paragraph{Hierarchical unconditionality.}
Consider the space $\R^N$, where the set of the coordinates is split into $b$
blocks of sizes $s_1,\ldots,s_b$; for $x\in\R^N$ we denote by
$x[j]\in\R^{s_j}$ the $j$-th block of $x$.

We have the following generalization of Theorem~\ref{thm_uncond}.
\begin{statement}
\label{stm_hierblock}
    Let $\xi=(\xi_1,\dots,\xi_N)=(\xi[1],\dots,\xi[b])$ be a random vector in $\R^N$, $N=s_1+\ldots+s_b$, that is block-unconditional, i.e.
    \begin{enumerate}
        \item $\xi[j]$ is unconditional for $j=1,\ldots,b$;
        \item $\Law(\xi[1],\ldots,\xi[b])=\Law(\pm\xi[1],\ldots,\pm\xi[b])$ for
            any choice of  signs.
    \end{enumerate}
    Suppose that $1< q\le 2$ and $1\le \E|\xi_i|^q< \infty$ for all $i$. Then
    $$
    d_n^\avg(\xi,\ell_q^N)_q \ge c(q,\eps)N^{1/q},\quad\mbox{if }n\le N(1-\eps).
    $$
\end{statement}
\begin{proof}
    The proof is similar to that  of Theorem~\ref{thm_uncond}. W.l.o.g.
    $\E|\xi_i|^q=1$. We define a biorthogonal vector
    $\eta=(\eta_1,\ldots,\eta_N)$, where
    $\eta_i:=|\xi_i|^{q-1}\sign\xi_i$, and then we bound the weak moment of $\eta$. Let
    $|v|=1$,  we have
    $$
    \|\langle v,\eta\rangle\|_{q'} = \left\|\sum_{j=1}^b
    |v[j]|\cdot \zeta^j \right\|_{q'},\quad \zeta^j :=
    \langle\frac{v[j]}{|v[j]|},\eta[j]\rangle.
    $$
    Using
    the unconditionality in each block, we obtain by Lemma~\ref{lem_rud} that
    $$
    \|\zeta^j\|_{q'} \le C_q\max_i \|\eta_i\|_{q'} = C_q.
    $$
    Next we use unconditionality between blocks:
    $$
    \|\, |v[1]|\cdot\zeta^1 + \ldots + |v[b]|\cdot\zeta^b \|_{q'}
    \le C_1(q)\max_j\|\zeta^j\|_{q'} \le C_2(q).
    $$
    It remains to apply Theorem~\ref{th_sp}.
\end{proof}

One can identify $N_1\times N_2$ matrices $M$ with vectors of size $N=N_1N_2$ that
consist of $b=N_2$ blocks; each block is some column of $M$.
Consider the action of the group
$G=S_{N_1}\times S_{N_2}\times \mathbb Z_2^{N_1}\times \mathbb Z_2^{N_2}$
on the matrix space $\R^{N_1\times N_2}$:
\begin{equation}\label{matrix_action}
(\pi^1,\pi^2,\tau^1,\tau^2)M = M',\quad M'_{i,j} :=
    (-1)^{\tau^1_i}(-1)^{\tau^2_j} M_{\pi^1(i),\pi^2(j)}.
\end{equation}
In other words, we permutate rows and columns and change signs of entire
rows/columns.

Note that permutations of rows/columns define a transitive (coordinate-wise)
action, so the usual reasoning applies.
\begin{statement}
\label{stm_matrices}
    Let $K$ be a set in the space of real $N_1\times N_2$ matrices  that is
    invariant under the action~\eqref{matrix_action}, i.e.  permutations of the sets of rows, the set of columns, and
    changes of sign of entire row or column. Then $K$ is rigid in $\ell_q^N$,
    $1<q\le 2$, i.e. $d_n(K,\ell_q^N) \ge c(q,\eps)\sup_{x\in K}\|x\|_q$, if $n\le N(1-\eps)$.
\end{statement}
Many sets of matrices satisfy this property, e.g. balls in
unitary--invariant norms, or balls in operator norms
between spaces $(\R^{N_2},\|\cdot\|)$ and $(\R^{N_1},\|\cdot\|)$ with
symmetric norms.

We can apply  our techniques  to the following  family of polytopes 
considered by A.~Vasilyeva in the context of the widths for mixed
norms~\eqref{mixed_lp_norm}.
Let $V_{k_1,k_2}^{N_1,N_2} := \mathrm{conv}\{gM \colon g\in
G\}$, where $M$ is the matrix such that $M_{i,j}=1$ for $1\le i\le k_1,
\,1\le j\le k_2$; and $M_{i,j}=0$ otherwise. The
 widths of these polytopes
   in mixed norms $\ell_{q_1,q_2}^{N_1,N_2}$, $q_1,q_2\ge 2$  were estimated  in ~\cite{Vas13}. We easily obtain from Statement \ref{stm_matrices} the estimate for $1<q\le 2,$ namely: 
   \begin{corollary}
    The set $V_{k_1,k_2}^{N_1,N_2}$ is rigid in $\ell_q^N$, $1<q\le 2$.
\end{corollary}

We note that the case $q=1$ is not covered by Statement \ref{stm_hierblock} and its corollaries.

\paragraph{Isotropic sets and random vectors.}
Let us give some corollaries of Theorem~\ref{th_sp} that hold for all
$1<q<\infty$.

It is known that the weak moments of all orders are equivalent for log--concave
random vectors: $\sigma_{p_2}(\eta)\le c(p_2/p_1)\sigma_{p_1}(\eta)$, $1\le
p_1<p_2<\infty$, see~\cite[Th.2.4.6]{BGVV}.

A random vector $\xi=(\xi_1,\dots, \xi_N)$ is called isotropic if it is
biorthogonal to himself, i.e. $\E\xi_i^2=1$, $\E\xi_i\xi_j=0$ for all $i\ne j$.

\begin{corollary}
    Let $\xi$ be an isotropic log--concave random vector in $\R^N$. Then for any
    $q\in(1,\infty)$ and $n\le N(1-\eps)$, $\eps\in(0,1)$, we have
    $$
    d^\avg_n(\xi,\ell_q^N)_q \ge c(q,\eps)N^{1/q}.
    $$
    Therefore, for any convex body $K\subset\R^N$ such that
    uniformely distributed random vector in $K$ is isotropic,
    $$
    d_n(K,\ell_q^N) \ge c(q,\eps) N^{1/q}.
    $$
    In particular, for any orthogonal operator $U\in \mathrm{O}(N)$ we have
    $$
    d_n(UB_\infty^N,\ell_q^N) \ge c(q,\eps)N^{1/q}.
    $$
\end{corollary}

\paragraph{Remarks.}
We note that Theorem~\ref{th_sp} is unitary invariant: if $\xi$ and
$\eta$ are biorthogonal, then $U\xi$ and $U\eta$ are biorthogonal too and
$\sigma_q(\eta)=\sigma_q(U\eta)$ for any $U\in \mathrm{O}(N)$. Hence one can replace
$\xi$ by $U\xi$ in the statements proved above.

We think that it will be interesting to obtain further generalizations of
Theorem~\ref{th_sp}.
We outline one such result. 
A random vector $\zeta$ in $\R^N$ defines the family of norms
$$
\|x\|_{L_p(\zeta)} := (\E|\langle x,\zeta\rangle|^p)^{1/p}.
$$
The unit ball of the dual space $L_p^*(\zeta)$ is called the $L_p$-centroid body
$Z_p(\zeta)$; see \cite[Ch.5]{BGVV}.

\begin{statement}
    Let $1<q<\infty$ and let $\xi,\eta,\zeta$ be some random vectors in $\R^N$ such that $\xi$ and
    $\eta$ are biorthogonal. Then for any $n<N$ we have
$$
    d_n^\avg(\xi,L_{q'}^*(\zeta))_q \cdot \sigma_{q'}(\eta) \cdot (\E|\zeta|^{q'})^{1/q'} \ge
    N-n.
$$
\end{statement}
The proof is completely analogous to that of Theorem~\ref{th_sp}.

\end{document}